\documentclass[12pt,a4paper]{article}
\usepackage[cp1251]{inputenc}
\usepackage[T2A]{fontenc}
\usepackage[english]{babel}

\usepackage{amsmath,amsfonts,amssymb,mathrsfs,amsopn,indentfirst,amscd}

\usepackage{graphicx}
\usepackage{amsthm}

\usepackage{amssymb}
\usepackage{amsmath}
\newcommand{\ver}{{\rm ver}}
\newcommand{\vo}{{\rm vol}}

\newtheorem{theorem}{Theorem}

\newtheorem*{corollary*}{Corollary}

\begin{document}
\title{Five-dimensional Perfect Simplices
}
\author{Mikhail Nevskii\footnote{Department of Mathematics,
              P.G.~Demidov Yaroslavl State University, Sovetskaya str., 14, Yaroslavl, 150003, Russia 
              orcid.org/0000-0002-6392-7618 
              mnevsk55@yandex.ru}        \and
        Alexey Ukhalov \footnote{ Department of Mathematics, P.G.~Demidov Yaroslavl State University, Sovetskaya str., 14, Yaroslavl, 150003, Russia 
              orcid.org/0000-0001-6551-5118 
              alex-uhalov@yandex.ru}
              }
\date{September 18, 2017}
\maketitle

\begin{abstract}
Let $Q_n=[0,1]^n$ be the unit cube in ${\mathbb R}^n$, $n \in {\mathbb N}$.
 For a nondegenerate simplex $S\subset {\mathbb R}^n$,
consider the value $\xi(S)=\min \{\sigma>0: Q_n\subset \sigma S\}$. 
Here $\sigma S$ is a homothetic image of $S$ with homothety center at the center of gravity of $S$ 
and coefficient of homothety $\sigma$. 
Let us introduce the value  $\xi_n=\min \{\xi(S): S\subset Q_n\}$. 
We call $S$ a perfect simplex if $S\subset Q_n$ and $Q_n$ is inscribed into the simplex $\xi_n S$. 
It is known that such simplices exist for $n=1$ and $n=3$. 
The exact values of $\xi_n$  are known for $n=2$ and in the case when there exist an  Hadamard matrix of order $n+1$; 
in the latter situation  $\xi_n=n$. 
In this paper we show that $\xi_5=5$ and $\xi_9=9$. We also describe infinite families of simplices  $S\subset Q_n$ such that $\xi(S)=\xi_n$ for $n=5,7,9$. 
The main result of the paper is the existence of perfect simplices in ${\mathbb R}^5$.


\end{abstract}

\section{Introduction}
\label{nev_uhl_p_intro}

Let us introduce the basic definitions.
We always assume that $n\in{\mathbb N}$. 
Element $x\in{\mathbb R}^n$ we present in component form as $x=(x_1,\ldots,x_n)$. 
Denote $Q_n:=[0,1]^n$, $Q_n':=[-1,1]^n$.

For a convex body $C\subset{\mathbb R}^n$,  by $\sigma C$ we mean  the result
of homothety of $C$ with center of homothety at the center of gravity of $C$ and coefficient $\sigma$. 

If $C$ is a convex polyhedron, then $\ver(C)$ 
means the set of vertices of $C$. 
We say that $n$-dimensional simplex $S$ 
is circumscribed around a convex body $C$ if $C\subset S$ and each $(n-1)$-dimensional face of $S$ contains a point of $C$. 
Convex polyhedron is inscribed into $C$ if each vertex of this polyhedron belongs to the boundary of $C$.

For a convex body  $C\subset{\mathbb R}^n$,
by $d_i(C)$ we denote  the length of a longest segment in $C$ parallel to the $i$th coordinate axis.
The  value $d_i(C)$ we call {\it the $i$th axial diameter of  $C$}. 
The notion of axial diameter was introduced by P.\,Scott  in \cite{nu_bib_11}, 
\cite{nu_bib_12}.

For nondegenerate simplex $S$ and convex body $C$ in ${\mathbb R}^n$,
we consider the value
$$\xi(C;S):=\min \{\sigma\geq 1: C\subset \sigma S\}.$$
Denote $\xi(S):=\xi(Q_n;S).$
The equality $\xi(C;S)=1$ is equivalent to the inclusion $C\subset S$.
For convex bodies $C_1, C_2\subset {\mathbb R}^n$, we define $\alpha(C_1;C_2)$  as the minimal $\sigma>0$ such that
$C_1$ belongs to a translate of $\sigma C_2$.     
Denote  $\alpha(C):=\alpha(Q_n;C)$. 
In this paper we compute the  values  $\xi(S)$ and $\alpha(S)$ for simplices $S\subset Q_n.$ 
We also consider  the value
$$
\xi_n:=\min \{ \xi(S): \, S \mbox{ --- $n$-dimensional simplex,} \, S\subset Q_n, \, \vo(S)\ne 0\}.
$$

Let us introduce basic Lagrange polynomials of $n$-dime\-nsio\-nal simplex. 
Let $\Pi_1({\mathbb R}^n)$ be the set of polynomials in $n$ real variables of degree $\leq 1$.
Consider a nondegenerate simplex $S$ in ${\mathbb R}^n$. Denote vertices of 
$S$ by $x^{(j)}=\left(x_1^{(j)},\ldots,x_n^{(j)}\right),$ 
$1\leq j\leq n+1,$ and build the matrix  
$$
{\bf A} :=
\left( \begin{array}{cccc}
x_1^{(1)}&\ldots&x_n^{(1)}&1\\
x_1^{(2)}&\ldots&x_n^{(2)}&1\\
\vdots&\vdots&\vdots&\vdots\\
x_1^{(n+1)}&\ldots&x_n^{(n+1)}&1\\
\end{array}
\right).
$$
Let $\Delta:=\det({\bf A})$, then $\vo(S)=\frac{|\Delta|}{n!}$.
Denote by $\Delta_j(x)$ the determinant  obtained from 
 $\Delta$ by changing the  $j$th row with  the row
$(x_1,\ldots,x_n, 1).$ 
Polynomials
$\lambda_j(x):=
\frac{\Delta_j(x)}{\Delta}$ 
from $\Pi_1({\mathbb R}^n)$ 
have the property
$\lambda_j\left(x^{(k)}\right)$ $=$ 
$\delta_j^k$, where $\delta_j^k$ is the  Kronecker delta. 
Coefficients of $\lambda_j$ form the $j$th column of ${\bf A}^{-1}$.
In the following we write
${\bf A}^{-1}$ $=(l_{ij})$, i.\,e.,~$\lambda_j(x)=
l_{1j}x_1+\ldots+
l_{nj}x_n+l_{n+1,j}.$

Each polynomial  $p\in \Pi_1({\mathbb R}^n)$ 
can be represented in the form
$$p(x)=\sum_{j=1}^{n+1} p\left(x^{(j)}\right)\lambda_j(x).$$
We call  $\lambda_j$ 
{\it basic Lagrange polynomials related to $S$}.
By taking  $p(x)=x_1, \ldots, x_n, 1$, we obtain
\begin{equation} \label{nev_uhl_lagr_sums}
\sum_{j=1}^{n+1} \lambda_j(x) x^{(j)}=x, \quad
\sum_{j=1}^{n+1} \lambda_j(x)=1.
\end{equation}
Therefore, the numbers $\lambda_j(x)$  
are barycentric coordinates of $x\in{\mathbb R}^n$ with respect to 
$S$. Simplex $S$ can be determined by each of the systems of inequalities
$\lambda_j(x)\geq 0$ or $0\leq \lambda_j(x)\leq 1$.

It is proved in \cite{nu_bib_1} that for $i$th axial diameter of simplex $S$ holds
\begin{equation}\label{nev_uhl_d_i_formula}
\frac{1}{d_i(S)}=\frac{1}{2}\sum_{j=1}^{n+1} \left|l_{ij}\right|. 
\end{equation}
There exist exactly one line segment in $S$ with the length $d_i(S)$ parallel to
the $x_i$-axis. The center of this segment coincides with the point
\begin{equation}\label{nev_uhl_max_center}
y^{(i)}= \sum_{j=1}^{n+1} m_{ij} 
x^{(j)}, \quad 
m_{ij}:=
\frac{\left|l_{ij}\right|}
{\sum\limits_{k=1}^{n+1}\left|l_{ik}\right|}.
\end{equation}                                        
Each $(n-1)$-face of $S$ 
contains at least one of the endpoints of this segment.
These results were generalized to the maximum line segment in $S$ 
parallel to an arbitrary vector $v\ne 0$. 
In~\cite{nu_bib_4} were obtained the formulae for the length and endpoints 
of such a segment via coordinates of vertices of $S$ and coordinates of $v$. 

From equality (\ref{nev_uhl_d_i_formula}) and properties of $l_{ij}$ (see \cite{nu_bib_3}, Chapter\,1) it follows that the value
$d_i(S)^{-1}$ 
is equal to the sum of the positive elements of the $i$th row of ${\bf A}^{-1}$ and simultaniously
is equal to the sum of the absolute values of the negative elements of this row. 

Note the following formulae for introduced numerical characteristics.
Let  $S$ be a nondegenerate simplex and $C$ be a convex body in ${\mathbb R}^n.$
It was shown in \cite{nu_bib_3} (see \S\,1.3) that in the case $C\not\subset S$ we have
\begin{equation}\label{nev_uhl_xi_c_formula}
\xi(C;S)=(n+1)\max_{1\leq k\leq n+1}
\max_{x\in C}(-\lambda_k(x))+1. 
\end{equation}
The condition
\begin{equation}\label{nev_uhl_circum_conv_cond}
\max\limits_{x\in C} \left(-\lambda_1(x)\right)=
\ldots=
\max\limits_{x\in C} \left(-\lambda_{n+1}(x)\right)
\end{equation}
is equivalent to the fact that simplex 
$\xi(S)S$ is circumscribed around $C$.
When  $C$ is a cube in ${\mathbb R}^n$,  
equality (\ref{nev_uhl_xi_c_formula}) can be written in the form
\begin{equation}\label{nev_uhl_xi_s_cub_formula}
\xi(S)=(n+1)\max_{1\leq k\leq n+1}
\max_{x\in \ver(C)}(-\lambda_k(x))+1,
\end{equation}
and (\ref{nev_uhl_circum_conv_cond}) can be replaced by the condition
\begin{equation}\label{nev_uhl_circum_cub_cond}
\max\limits_{x\in \ver(C)} \left(-\lambda_1(x)\right)=
\ldots=
\max\limits_{x\in \ver(C)} \left(-\lambda_{n+1}(x)\right).
\end{equation}

For  $C=Q_n$ formula (\ref{nev_uhl_xi_s_cub_formula}) 
was proved in \cite{nu_bib_1}. The statement  that  
(\ref{nev_uhl_circum_cub_cond}) holds if and only if $\xi(S)S$ is circumscribed 
around $Q_n$ follows from the results of \cite{nu_bib_1} and \cite{nu_bib_10}.

It was proved in \cite{nu_bib_3} (see \S\,1.4) that, for arbitrary convex body $C$ and nondegenerate 
simplex $S$ in ${\mathbb R}^n$, holds
\begin{equation}\label{nev_uhl_alpha_conv_formula}
\alpha(C;S)=
\sum_{j=1}^{n+1} \max_{x\in C} (-\lambda_j(x))+1. 
\end{equation}
If $C=Q_n$, then (\ref{nev_uhl_alpha_conv_formula}) is equivalent to
\begin{equation}\label{nev_uhl_alpha_d_i_formula}
\alpha(S)
=\sum_{i=1}^n\frac{1}{d_i(S)}.
\end{equation}

Equality (\ref{nev_uhl_alpha_d_i_formula}) was  established in \cite{nu_bib_10}. 
There are several interesting corollaries of this result.
For instance, we present here the formula for $\alpha(S)$ via coefficients 
of  $\lambda_j:$
\begin{equation}\label{nev_uhl_alpha_lij_formula}
\alpha(S)=\frac{1}{2}\sum_{i=1}^n\sum_{j=1}^{n+1} |l_{ij}|. 
\end{equation}
\noindent From (\ref{nev_uhl_alpha_lij_formula}) and properties of $l_{ij}$, it follows that 
 $\alpha(S)$ is equal to the sum of the positive elements of
upper $n$ rows of ${\bf A}^{-1}$ and simultaneously is equal to the sum 
of the absolute values of the negative elements of these rows. 

It is obvious that for convex body $C$ and simplex $S$ holds $\xi(C;S)\geq\alpha(C;S)$. The equality takes place only when simplex $\xi(S)S$ is circumscribed around $C$.

If $S\subset Q_n$, then $d_i(S)\leq 1$. Applying (\ref{nev_uhl_alpha_d_i_formula}) 
for this case we have
\begin{equation}\label{nev_uhl_xi_geq_alpha_geq_n}
\xi(S)\geq \alpha(S)\geq n.
\end{equation}
Consequently,
$\xi_n\geq n.$ 
By 2009, the first author obtained that
$$\xi_1=1,\quad\xi_2=\frac{3\sqrt{5}}{5}+1=2.34\ldots,\quad\xi_3=3,\quad 
4\leq \xi_4\leq \frac{13}{3}=
4.33\ldots,$$

$$5\leq \xi_5\leq 5.5,\quad 6\leq \xi_6\leq 6.6,\quad \xi_7=7.$$
If  $n>2$, then
\begin{equation}\label{nev_uhl_xi_leq_n23n1}
\xi_n\leq\frac{n^2-3}{n-1}. 
\end{equation}
Hence, for any $n$, we have $n\leq \xi_n<n+1$. 
Inequality
(\ref{nev_uhl_xi_leq_n23n1}) was established by calculations for the simplex $S^*$ 
with vertices
$(0,1,\ldots,1)$, 
$(1,0,\ldots,1)$, 
$\ldots$, $(1,1,\ldots,0)$, 
$(0,0,\ldots,0)$ 
(see \cite{nu_bib_2}, \cite{nu_bib_3}, \S\,3.2).
If $n>2$, then
$\xi(S^*)=\frac{n^2-3}{n-1},$
which gives (\ref{nev_uhl_xi_leq_n23n1}). 
If  $n\geq 3$, then  $S^*$ has the following property
(see \cite{nu_bib_7}, Lemma~3.3): replacement of an arbitrary vertex of  $S^*$ by 
any point of $Q_n$ decreases the volume of the simplex. 
For $n=2,3,4$ (and only in these cases),
$S^*$ is a simplex of maximum volume in $Q_n.$ If $n\geq 2$, then
$d_i(S^*)=1$, 
therefore $\alpha(S^*)=n.$ 
If $n=3$, then $\alpha(S^*)=\xi(S^*)$; for $n>3$ we have
$\alpha(S^*)<\xi(S^*)$.
  
If  $n+1$ is an Hadamard number, i.e., there exist an Hadamard matrix of the order $n+1$, 
then  $\xi_n=n$ (see section \ref{nev_uhl_p_adamarcase}). 
In this and only this case there exist a regular simplex $S$ inscribed into  $Q_n$ such that  vertices of $S$ coincide with vertices of $Q_n$ (\cite{nu_bib_7}, Theorem~4.5).
For such a simplex, we have  $\xi(S)=n$. 
It follows from (\ref{nev_uhl_xi_geq_alpha_geq_n}) that $\alpha(S)=n$, and  (\ref{nev_uhl_alpha_d_i_formula}) gives  $d_i(S)=1$. 

Since 2011, the first author supposed that the equality  $\xi_n=n$ holds
only if  $n+1$ is an Hadamard number. In this paper we will demonstrate that it is not so.
The smallest $n$ such that  $n+1$ is not an Hadamard number, while  $\xi_n=n$,
is equal to $5$; the next number with such a property is $9$. 
Below we study simplices  $S$ such that $S\subset Q_n\subset nS$,
for odd $n$, $1\leq n\leq 9$. Since $\xi_2=2.34...>2$, there no exist such 
a simplex in the case $n=2$.
For even numbers $n\geq 4$, the problem of existence of simplices
with the condition  $S\subset Q_n\subset nS$ is unsolved. 
In  \cite{nu_bib_5} the authors proved that $\xi_4\leq \frac{19+5\sqrt{13}}{9}=4.1141\ldots$, and conjectured that  $\xi_4=\frac{19+5\sqrt{13}}{9}$.

A nondegenerate simplex $S$ in ${\mathbb R}^n$  we will call  {\it a perfect simplex
with respect to an $n$-dimensional cube $Q$} if $S\subset Q\subset \xi_n S$ and cube $Q$
is inscribed into simplex $\xi_n S$, i.\,e., the boundary of $\xi_n S$ contains all the vertices of
$Q$. If simplex is perfect with respect of the cube $Q_n$, then  we  shortly call such a simplex 
 {\it perfect}. 

By the moment of submitting the paper, only three values of $n$,
 such that perfect simplices  in ${\mathbb R}^n$ exist, are known to the authors.
These numbers are $1, 3 \text{ and } 5$. In all these cases we have $\xi_n=n$.
The case $n=1$ is trivial. The unique (up to similarity) three-dimensional 
perfect simplex  is described in section~\ref{nev_uhl_p_n13}. 
The main result of this paper is the existence of perfect simpices in ${\mathbb R}^5$, see section~\ref{nev_uhl_p_xi5is5}.
Moreother, in section~\ref{nevuhl_perf_fam5} we describe the whole  family of $5$-dimensional perfect simplices.

\section{Supporting Vertices of the Cube}
\label{nev_uhl_p_lem_vertices}

Let $S$ be an $n$-dimensional nondegenerate simplex with vertices $x^{(1)},$
$\ldots,$ $x^{(n+1)}$ and let $\lambda_1$, $\ldots$, $\lambda_{n+1}$ be basic Lagrange polynomials of $S$. 
We define the $j$th facet of $S$ as its $(n-1)$-dimensional face which
does not contain the vertex $x^{(j)}$. In other words, $j$th facet of  $S$ is an $(n-1)$-face cointained in hyperplane  $\lambda_j(x)=0$. 
By the $j$th facet of simplex $\sigma S$, we mean the facet parallel to the $j$th facet of $S$.

\begin{theorem}\label{nev_uhl_lem_vertices}
Let $S\subset Q_n$. The vertex $v$ of $Q_n$ belongs to the  $j$th facet of 
simplex $\xi(S)S$  if and only if
\begin{equation}\label{nev_uhl_cub_lambd}
-\lambda_j(v)=\max_{1\leq k\leq n+1, x\in \ver(Q_n)} (-\lambda_k(x)). 
\end{equation}
\end{theorem}

\begin{proof} Denote $\xi:=\xi(S)$. First note that
the equation of hyperplane containing the $j$th facet of  simplex $\xi S$ can be written in the form
\begin{equation}\label{nev_uhl_cub_lambd_1xin1}
\lambda_j(x)=\frac{1-\xi}{n+1}. 
\end{equation}

Indeed, this equation has a form $\lambda_j(x)=a$. 
Remind that $\lambda_j(x)$ is a linear function of $x$. Since the center of gravity of 
$S$ is contained in the hyperplane $\lambda_j(x)=\frac{1}{n+1}$ and 
the $j$th facet of $S$ belongs to the hyperplane $\lambda_j(x)=0$, we have
$$
\frac{a-\frac{1}{n+1}}{0-\frac{1}{n+1}}=\xi.
$$
Hence, $a=\frac{1-\xi}{n+1}$.

Let us give another proof of this fact.
Let  $c$ be the center of gravity of $S$, $y$ be the vertex of  $S$ not coinciding with 
$x^{(j)}$, and $z$ be the homothetic image of the point $y$ with  coefficient of 
homothety $\xi$ and center of homothety at $c$. Then
$z-c=\xi(y-c)$ and $z=\xi y+(1-\xi)c$. 
The functional 
$\mu_j(x):= \lambda_j(x)-\lambda_j(0)$ is a linear (homogeneous and additive) 
functional on ${\mathbb R}^n$. Since $\lambda_j(y)=0$ and $\lambda_j(c)=\frac{1}{n+1},$
we have
$$\lambda_j(z)-\lambda_j(0)=\mu_j(z)=\mu_j(\xi y+(1-\xi)c)=\xi\mu_j(y)+(1-\xi)\mu_j(c)=$$
$$=\xi[\lambda_j(y)-\lambda_j(0)]+(1-\xi)[\lambda_j(c)-\lambda_j(0)]=
\frac{1-\xi}{n+1}-\lambda_j(0).$$
This implies $\lambda_j(z)=\frac{1-\xi}{n+1}$. 
The point $z$ belongs to the  $j$th facet of simplex
$\xi S$, therefore, the equation of hyperplane cointaining this facet has the form (\ref{nev_uhl_cub_lambd_1xin1}).

Now assume that $v\in\ver(Q_n)$ satisfy the equation  (\ref{nev_uhl_cub_lambd_1xin1}). Since
\begin{equation}\label{nev_uhl_xi_cub_lambda_formula}
\xi=(n+1)
\max_{1\leq k\leq n+1, x\in \ver(Q_n)} (-\lambda_k(x))+1 ,
\end{equation}
see (\ref{nev_uhl_xi_s_cub_formula}), the right part of (\ref{nev_uhl_cub_lambd}) is equal to 
$\frac{\xi-1}{n+1}$. Therefore, $v$ belongs to the  $j$th facet of $\xi S$. 
On the other hand, if $v$ belongs to the $j$th 
facet of $\xi S$, then $-\lambda_j(v)=\frac{\xi-1}{n+1}$. 
By (\ref{nev_uhl_xi_cub_lambda_formula}),  this is equivalent to (\ref{nev_uhl_cub_lambd}). 
This completes the proof.
\end{proof}


\section{On a Simplex Satisfying the Condition  \newline
$S\subset Q \subset n S$}
\label{nev_uhl_p_center}

In this section we denote by $c(D)$ the center of gravity of a convex body $D$.

\begin{theorem}\label{nev_uhl_theor_center}
Let  $Q$ be a nondegenerate parallelotope and $S$ is a nondegenerate simplex in ${\mathbb R}^n$.
If
$S\subset Q \subset nS$, then $c(S)=c(Q)$.
\end{theorem}

\begin{proof}
Any two nondegenerate parallelotopes in ${\mathbb R}^n$ are affine equivalent. 
Corresponding nondegenerate affine transformation maps a simplex to a simplex.
This transformation maps the center of gravity of each polyhedron into the center of gravity of
the image of this polyhedron.
Therefore, it is enough to prove the theorem in the case 
$Q=Q_n'=[-1,1]^n$. 
Namely, we will show that if $Q=Q_n'$ and conditions of the theorem are satisfied, then $c(S)=c(Q_n')=0$.

Let $x^{(j)}$ be vertices of $S$, and $\lambda_j$ $(1\leq j\leq n+1)$ be basic Lagrange polynomials of~$S$.

Since $S\subset Q_n' \subset nS$, we have $\xi(Q_n';S)\leq n$. 
For any simplex $T\subset Q$, holds $\xi(Q_n';T)\geq \alpha(Q_n';T)\geq n$. 
Hence, $\xi(Q_n';S)=\alpha(Q_n';S)=n$. 
It follows that simplex  $nS$ is circumscribed around $Q_n'$. 
So, the following equalities hold
$$
\max_{x\in\ver(Q_n')}(-\lambda_1(x))=
\ldots=
\max_{x\in\ver(Q_n')}(-\lambda_{n+1}(x)).
$$
In other words, $\max_{x\in\ver(Q_n')}(-\lambda_j(x))$ does not depend on $j$. 
From the equality
$$
n=\xi(Q_n';S)=(n+1)\max_{1\leq j\leq n+1, x\in\ver(Q_n')} (-\lambda_j(x))+1
$$
we obtain, for any $j$, 
\begin{equation}\label{nev_uhl_max_labd_n1n1}
\max\limits_{x\in\ver(Q_n')} (-\lambda_j(x))=\frac{n-1}{n+1}.
\end{equation}

Let  $l_{ij}$ be the coefficients of $\lambda_j$, i.\,e.,
$\lambda_j(x)=l_{1j}x_1+\ldots+l_{nj}x_n+l_{n+1,j}$.
Then
\begin{equation}\label{nev_uhl_max_sum_lij}
\max_{x\in\ver(Q_n')} (-\lambda_j(x))=\sum_{i=1}^n |l_{ij}|-l_{n+1,j}.
\end{equation}
Consider the value
$\max\limits_{x\in Q_n'} \lambda_j(x)$. 
It is easy to see that
\begin{equation}\label{nev_uhl_max_sum_lij_1}
\max_{x\in Q_n'} \lambda_j(x)=\sum_{i=1}^n |l_{ij}|+l_{n+1,j}.
\end{equation}
Since $x^{(j)}\in Q_n'$, we have
\begin{equation}\label{nev_uhl_max_q_labd1}
\max_{x\in Q_n'} \lambda_j(x)\geq \lambda_j\left(x^{(j)}\right)=1. 
\end{equation}

Using (\ref{nev_uhl_max_labd_n1n1})--(\ref{nev_uhl_max_q_labd1}), we obtain
$$2\cdot l_{n+1,j}=
\left[ 
\sum_{i=1}^n |l_{ij}|+l_{n+1,j}\right]-
\left[ 
\sum_{i=1}^n |l_{ij}|-l_{n+1,j}\right]=$$
$$=\max_{x\in Q_n'} \lambda_j(x)-  
\max_{x\in\ver(Q_n')} (-\lambda_j(x))\geq
1-\frac{n-1}{n+1}=\frac{2}{n+1}.$$
Consequently, for any  $j$, the following inequality holds:
\begin{equation}\label{nev_uhl_lam0}
\lambda_j(0)=l_{n+1,j}\geq \frac{1}{n+1}. 
\end{equation}
The numbers $\lambda_j(0)$ are the barycentric coordinates of the point $x=0$ (see (\ref{nev_uhl_lagr_sums})),
hence,
\begin{equation}\label{nev_uhl_sumlam0}
\sum_{j=1}^{n+1}\lambda_j(0)=1. 
\end{equation}
If, for some  $j$, inequality (\ref{nev_uhl_lam0}) is strict, then the left part of
(\ref{nev_uhl_sumlam0}) is strictly greater than the right part. This is not possible, therefore,
\begin{equation}\label{nev_uhl_lam0_1n1}
\lambda_j(0)=\frac{1}{n+1}, \quad 1\leq j\leq n+1. 
\end{equation}
Thus, we have
$$c(S)=\frac{1}{n+1}\sum_{j=1}^{n+1}x^{(j)}=\sum_{j=1}^{n+1}\lambda_j(0)x^{(j)}=0.$$

In addition, note that for any  $j$ we have 
$\max\limits_{x\in Q_n'} \lambda_j(x)=1$. 
Indeed,  from (\ref{nev_uhl_max_sum_lij}), (\ref{nev_uhl_max_sum_lij_1}), and (\ref{nev_uhl_lam0_1n1}) it follows
$$\max_{x\in Q_n'} \lambda_j(x)=\sum_{i=1}^n |l_{ij}|+l_{n+1,j}=
\max\limits_{x\in\ver(Q_n')} (-\lambda_j(x))+2l_{n+1,j}=
$$
$$
=\frac{n-1}{n+1}+
\frac{2}{n+1}=1.
$$
Hence, if 
$S\subset Q_n' \subset nS$, then the cube $Q_n'$ lies in the intersection of half-spaces
$\lambda_j(x)\leq 1$. 
The theorem is proved.
\end{proof}


\section{The Case when $n+1$ is an Hadamard Number}
\label{nev_uhl_p_adamarcase}

A nondegenerate $(m\times m)$-matrix
${\bf H}_m$  is called {\it an Hadamard matrix of order $m$} if its elements are equal 
to $1$ or $-1$ and
$${\bf H}_m^{-1}=
\frac{1}{m}{\bf H}_m^T.$$
Some information on Hadamard matrices can be found in \cite{nu_bib_6}.
It is known that if
${\bf H}_m$ exists, then $m=1,$ $m=2$ or $m$ is a multiple of $4$. 
It is established that ${\bf H}_m$ exists for infinite set of numbers $m$ 
of the form $m=4k$, including powers of two $m=2^l$. 
By 2008, the smallest number $m$ for which it is unknown whether there is an Hadamard 
matrix of order $m$ was equal to $668$. 
We call a natural number $m$ {\it an Hadamard number} if there exist an Hadamard matrix of order  $m$.

The regular simplex $S$, inscribed into $Q_n$ in such a way that vertices of $S$ are situated 
in vertices of $Q_n$, exists if and only if $n+1$ is an Hadamard number (see \cite{nu_bib_7}, Theorem~4.5). 
If  $n+1$ is an Hadamard number, then $\xi(S)=n$ and, therefore, $\xi_n=n$. The latter statement
was proved by two different methods in the paper \cite{nu_bib_2} and in the book \cite{nu_bib_3}, \S\,3.2.

Here we give yet another proof of this fact. This new proof differs from the proofs given in \cite{nu_bib_5} and~\cite{nu_bib_6} and directly utilizes properties of Hadamard matrices. 
We also establish some other properties of regular simplex $S$.

\begin{theorem}\label{nev-uhl-adamartheorem}
Let $n+1$ be an Hadamard number and $S$ be a regular simplex inscribed in  $Q_n$. Then
$\xi_n=\xi(S)=n$. 
\end{theorem}

\begin{proof}
Taking into consideration similarity, we can prove the statement for the cube $Q^\prime_n:=[-1,1]^n$.
Since $n+1$ is an Hadamard number, there exist a {\it normalized Hadamard matrix} of order $n+1$, 
i.\,e., such an Hadamard matrix that its first row and first column consist of  $1$'s (see \cite{nu_bib_9}, Chapter\,14).
Let us write rows of this matrix in inverse order:
$${\bf H} =
\left( \begin{array}{ccccc}
1&1&1&\ldots&1\\
\ldots&\ldots&\ldots&\ldots&1\\
\ldots&\ldots&\ldots&\ldots&1\\
\ldots&\ldots&\ldots&\ldots&\ldots\\
\ldots&\ldots&\ldots&\ldots&1\\
\end{array}
\right).$$
The obtained matrix {\bf H} also is an Hadamard matrix of order $n+1$. 
Consider the simplex $S^\prime$ with vertices 
formed by first $n$ numbers in rows of $\bf H$.

It is clear that all vertices of $S^\prime$ are also the vertices of $Q^\prime_n$, therefore,
the simplex is inscribed into the cube. Let us show that $S^\prime$ is a regular simplex
and the lengths of its edges are equal to $\sqrt{2(n+1)}$. 
Let  $a$, $b$ be two different rows of ${\bf H}$. All elements of 
$\bf H$ are $\pm 1$, hence, we have 
$\|a\|^2=\|b\|^2=n+1$. 
Rows of an Hadamard matrix are mutually orthogonal, therefore,
$$
\|a-b\|^2=(a-b,a-b)=\|a\|^2+\|b\|^2-2(a,b)=2(n+1).
$$
Denote by $u$ and $w$ the vertices of  $S^\prime$ obtained from  $a$ and $b$
respectively by removing the last component. 
This last component is equal to 1. It follows that $n$-dimensional length of the vector $u-w$ is 
equal to $(n+1)$-dimensional length  of the vector $a-b$, i.\,e., $\|u-w\|=\sqrt{2(n+1)}$. 

Denote by $\lambda_j$ basic Lagrange polynomials of $S^\prime$. 
Since ${\bf H}^{-1}=
\frac{1}{n+1}{\bf H}^T$, 
the coefficients of 
$(n+1)\lambda_j$ are situated in rows of  ${\bf H}$. 
All constant terms of these polynomials stand in the last column of ${\bf H}$. 
Consequently, they are equal to $1$. It means that constant terms of polynomials
$-(n+1)\lambda_j$ are equal to $-1$.
By this reason, for any $j=1,\ldots,n+1$, 
$$(n+1)\max_{x\in\ver(Q^\prime_n)}(-\lambda_j(x))=n-1.$$
The coefficients of polynomials $-(n+1)\lambda_j$ are equal to $\pm 1$, 
therefore, for any $j$, the vertex  $v$ of $Q^\prime_n$, such that $(n+1)(-\lambda_j(v))=n-1$,
is unique.
Indeed, $v=(v_1,\ldots,v_n)$ is defined by equalities $v_i=- \operatorname{sign} l_{ij}$, where  $l_{ij}$ are the  coefficients of $\lambda_j$. 

Now let us find $\xi(Q^\prime_n; S^\prime)$ using formula (\ref{nev_uhl_xi_s_cub_formula}) 
for $C=Q^\prime_n$.
We have
$$\xi(Q^\prime_n;S^\prime)=(n+1)\max_{1\leq j\leq n+1}
\max_{x\in \ver(Q^\prime_n)}(-\lambda_j(x))+1=n-1+1=n.
$$

Consider the similarity transformation which maps $Q^\prime_n$ into $Q_n$.
This transformation also maps $S^\prime$ into a simplex inscribed into $Q_n$. 
Denote by  $S$ the image of $S^\prime$.
Obviously, $\xi(S)=\xi(S^\prime;Q^\prime_n)=n$.
It follows that, if $n+1$ is an Hadamard number, then $\xi_n\leq n$. 
As we know, for any $n$, $\xi_n\geq n$ (see (\ref{nev_uhl_xi_geq_alpha_geq_n})). 
Hence, $\xi_n=\xi(S)=n.$ 
Since the coefficient of similarity for mapping $Q^\prime_n$
to $Q_n$ is equal to $\frac{1}{2}$, the length of any edge of $S$ is equal to $\sqrt{   \frac{n+1}{2}  }$.

From the condition
$$
\max\limits_{x\in \ver(Q^\prime_n)} \left(-\lambda_1(x)\right)=
\ldots=
\max\limits_{x\in \ver(Q^\prime_n)} \left(-\lambda_{n+1}(x)\right)=\frac{n-1}{n+1},
$$
it follows that simplex
$nS^\prime$ 
is circumscribed around the cube $Q^\prime_n$. 
From the said above it also follows that
each $(n-1)$-face of $nS^\prime$ 
contains only one vertex of $Q^\prime_n$.
This means that simplex $nS$ is circumscribed around the cube $Q_n$ in the same way.
Since $nS$ is circumscribed around $Q_n$, we have $\alpha(S)=\xi(S)=n$
and $d_i(S)=1$. These equalities could be also obtained from inequalities
$n\leq \alpha(S)\leq \xi(S)$ and equality $\xi(S)=n$. 

Equalities  $d_i(S)=1$ and $\alpha(S)=n$ can be also derived in another way.
Note that the regular simplex $S$ inscribed into $Q_n$ has the maximum volume 
among all the simplices in  $Q_n$, see Theorem~2.4 in \cite{nu_bib_7}.
Additionally, for any simplex of maximum volume in $Q_n$, all the  axial diameters are equal to $1$. 
The latter property of a maximum volume simplex in $Q_n$ was established by Lassak in \cite{nu_bib_8}.
This fact also can be obtained from  (\ref{nev_uhl_alpha_d_i_formula}) (see \cite{nu_bib_10} and \cite{nu_bib_3}, \S\,1.6]). 

The theorem is proved.
\end{proof}


\section{Perfect Simplices in ${\mathbb R}^1$ and ${\mathbb R^3}$}
\label{nev_uhl_p_n13}

The case  $n=1$ is very simple. 
For the segment $S=[0,1]$,  we have $S= Q_1$.
Therefore, $\xi_1= 1$ and $S$ is the unique perfect simplex.
The equality $\xi_1= 1$ could be also obtained  from Theorem~\ref{nev-uhl-adamartheorem}.

Consider the case $n=3$. The conditions of 
Theorem~\ref{nev-uhl-adamartheorem} are satisfied, hence, $\xi_3=3$. There exist a regular simplex $S$ inscribed into $Q_3$ such that $Q_3\subset 3S$.
As an example, we consider the simplex $S_1$ with vertices
$(0, 0, 0)$,
$(1, 1, 0)$,
$(1, 0, 1)$,
$(0, 1, 1)$.
Matrices  ${\bf A}$ and ${\bf A}^{-1}$ for $S_1$ have the form
$$
{\bf A}=
\left(
\begin{array}{cccc}
 0 & 0 & 0 & 1 \\
 1 & 1 & 0 & 1 \\
 1 & 0 & 1 & 1 \\
 0 & 1 & 1 & 1 \\
\end{array}
\right),\quad
{\bf A}^{-1}=\frac{1}{2}
\left(
\begin{array}{rrrr}
 -1 & 1 & 1 & -1 \\
 -1 & 1 & -1 & 1 \\
 -1 & -1 & 1 & 1 \\
 2 & 0 & 0 & 0 \\
\end{array}
\right).
$$
Using (\ref{nev_uhl_d_i_formula}), we get $\frac{1}{d_i(S)}=\frac{1}{2}\sum_{j=1}^{4} \left|l_{ij}\right|=1$, i.\,e., $d_i(S_1)=1$.
It follows from (\ref{nev_uhl_alpha_d_i_formula}) that
$\alpha(S_1)=\sum_{i=1}^3\frac{1}{d_i(S_1)} = 3$.
Basic Lagrange polynomials of $S_1$ are
\begin{align*}
\lambda_1(x)&=  \frac{1}{2} \left(-x_1-x_2-x_3+2\right), &
\lambda_2(x)&= \frac{1}{2} \left(x_1+x_2-x_3\right), \\
\lambda_3(x) &= \frac{1}{2} \left(x_1-x_2+x_3\right),  &
 \lambda_4(x) &= \frac{1}{2} \left(-x_1+x_2+x_3\right).
\end{align*}
By (\ref{nev_uhl_xi_s_cub_formula}), we have
$$
\xi(S_1)=4\max_{1\leq k\leq 4} \,\max_{x\in \ver(Q_3)}(-\lambda_k(x))+1.
$$
Substituting vertices of  $Q_3$ into basic Lagrange polynomials we obtain 
\begin{equation}\label{nev_uhl_max12}
\max_{1\leq k\leq 4} \,\max_{x\in \ver(Q_3)}(-\lambda_k(x))=\frac{1}{2}.
\end{equation}
Hence, $\xi(S_1)=4 \cdot  \frac{1}{2} +1 =3$.
This implies $\xi(S_1) = \xi_3$.
For each $k$, maximum in (\ref{nev_uhl_max12}) is achieved only for one vertex of $Q_3$:
\begin{equation}\label{nev_uhl_lam3vertcond}
-\lambda_1(1,1,1) = -\lambda_2(0,0,1) = -\lambda_3(0,1,0)= -\lambda_4(1,0,0) = \frac{1}{2}.
\end{equation}
It follows from Theorem~\ref{nev_uhl_lem_vertices} and (\ref{nev_uhl_lam3vertcond}) that
the extremal vertices of the cube, i.\,e., the vertices
$(1,1,1)$, $(0,0,1)$, $(0,1,0)$, $(1,0,0)$ belong to the facets of the simplex $3S_1$. Every facet of $3S_1$ contains only one vertex of $Q_3$. 
The rest vertices of the cube belong to the interior of $3S_1$.
Thus, though for $S_1$ we have $S_1 \subset Q_3 \subset 3S_1$,  the simplex $S_1$ is not perfect.

From (\ref{nev_uhl_max_center}), we obtain that three maximum segments in $S_1$  parallel to coordinate axes are intersected at the center of the cube.

\begin{figure}
\includegraphics[width=1.0\textwidth]{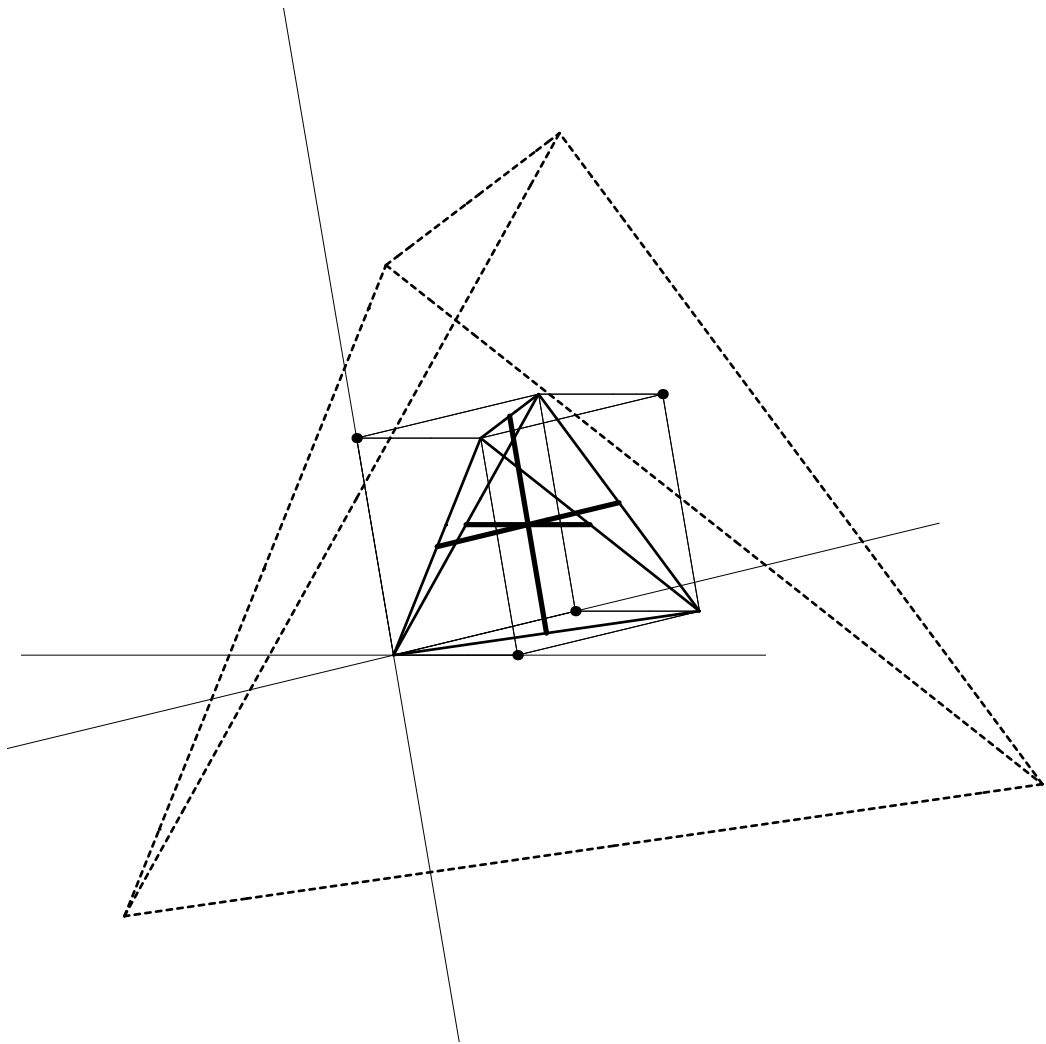}
\caption{Simplex  $S_1$}
\label{fig:nev_ukl_simplex_s1}      
\end{figure}

\begin{figure}
\includegraphics[width=1.0\textwidth]{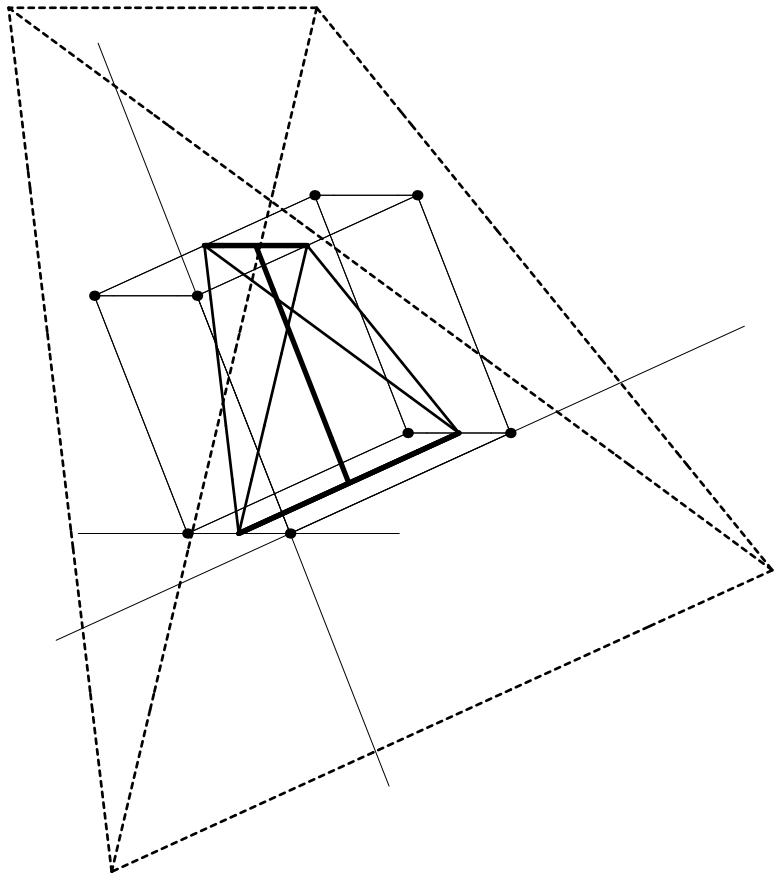}
\caption{Simplex  $S_2$}
\label{fig:nev_ukl_simplex_s2}      
\end{figure}

Now consider the simplex  $S_2$ with vertices
$\left(\frac{1}{2},0,0\right)$,  
$\left(\frac{1}{2},1,0\right)$, 
$\left(0,\frac{1}{2},1\right)$,
$\left(1,\frac{1}{2},1\right)$.
For this simplex,
$$
{\bf A}=
\left(
\arraycolsep=5pt\def\arraystretch{1.2}
\begin{array}{cccc}
 \frac{1}{2} & 0 & 0 & 1 \\
 \frac{1}{2} & 1 & 0 & 1 \\
 0 & \frac{1}{2} & 1 & 1 \\
 1 & \frac{1}{2} & 1 & 1 \\
\end{array}
\right),\quad
{\bf A}^{-1}=\frac{1}{2}
\left(
\arraycolsep=4pt\def\arraystretch{1.0}
\begin{array}{rrrr}
 0 & 0 & -2 & 2 \\
 -2 & 2 & 0 & 0 \\
 -1 & -1 & 1 & 1 \\
 2 & 0 & 1 & -1 \\
\end{array}
\right).
$$
We have $\frac{1}{d_i(S_2)}=\frac{1}{2}\sum_{j=1}^{4} \left|l_{ij}\right|=1$.  
Consecuently $d_i(S_2)=1$ and $\alpha(S_2) =\sum_{i=1}^3\frac{1}{d_i(S_2)} = 3$.
Basic Lagrange polynomials of $S_2$ are
\begin{align*}
\lambda_1(x)&=  -x_2-\frac{1}{2}\, x_3+1, &
\lambda_2(x)&= x_2-\frac{1}{2}\, x_3, \\
\lambda_3(x) &= \frac{1}{2} \left(-2 x_1+x_3+1\right),  &
 \lambda_4(x) &= \frac{1}{2} \left(2 x_1+x_3-1\right).
\end{align*}
The formula (\ref{nev_uhl_xi_s_cub_formula}) takes the form
$$
\xi(S_2)=4\max_{1\leq k\leq 4} \,\max_{x\in \ver(Q_3)}(-\lambda_k(x))+1.
$$
Substituting vertices of $Q_3$ into polynomials we obtain
\begin{equation} \label{nev_uhl_max12vert}
\begin{array}{l}
- \lambda_1(0,1,1) = - \lambda_1(1,1,1) = - \lambda_2(0,0,1) = - \lambda_2(1,0,1) = \\
- \lambda_3(1,0,0) = - \lambda_3(1,1,0) = - \lambda_4(0,0,0) = - \lambda_4(0,1,0) = \frac{1}{2}.
\end{array}
\end{equation}
Hence, $\max_{k,\,x\in \ver(Q_3)}(-\lambda_k(x))=\frac{1}{2}$ and 
$\xi(S_2)=4 \cdot  \frac{1}{2} +1 =3=\xi_3$.
In this case all the vertices of $Q_3$ are extremal.
It follows from Theorem~\ref{nev_uhl_lem_vertices} and (\ref{nev_uhl_max12vert}) that
each vertex of  $Q_3$ belongs to the boundary of $3S_2$. 
Therefore, $S_2$, unlike $S_1$, is a perfect simplex.

It is proved in  \cite{nu_bib_3} that $S_1$ and $S_2$  
are the only simplices in $Q_3$ (up to similarity) with property $\xi(S)=\xi_3=3$.
Each perfect simplex in ${\mathbb R}^3$ is similar to  $S_2$.

The simplices $S_1$, $3S_1$, $S_2$, $3S_2$ are shown in Fig.~\ref{fig:nev_ukl_simplex_s1} and Fig.~\ref{fig:nev_ukl_simplex_s2}. 
Vertices of $Q_3$ situated on the boundary of $3S_1$ and $3S_2$ are marked with bold points.
Bold lines mark the segments corresponding to the axial diameters.

\section{The Exact Value of  $\xi_5$}
\label{nev_uhl_p_xi5is5}

\begin{theorem}\label{nev-uhl-ksi5is5}
There exist simplex $S \subset Q_5$ such that  simplex $5S$ is circumscribed 
around  $Q_5$ and boudary  of 
$5S$ contains all vertices of the cube. 
\end{theorem}
\begin{proof}
Consider the simplex $S\subset {\mathbb R}^5$ with vertices
$\left(\frac{1}{2},1,\frac{1}{3},1,1\right)$,
$\left(\frac{1}{2},0,\frac{1}{3},1,1\right)$, 
$\left(\frac{1}{2},\frac{1}{2},\frac{1}{3},0,1\right)$,
$\left(\frac{1}{2},\frac{1}{2},0,\frac{1}{3},0\right)$,
$\left(0,\frac{1}{2},1,\frac{1}{3},0\right)$,
$\left(1,\frac{1}{2},1,\frac{1}{3},0\right)$.
Obviously, $S\subset Q_5$. 
Matrices ${\bf A}$ and  ${\bf A}^{-1}$ for the simplex  $S$ have the form
$$
{\bf A}=
\left(
\arraycolsep=5pt\def\arraystretch{1.2}
\begin{array}{cccccc}
 \frac{1}{2} & 1 & \frac{1}{3} & 1 & 1 & 1 \\
 \frac{1}{2} & 0 & \frac{1}{3} & 1 & 1 & 1 \\
 \frac{1}{2} & \frac{1}{2} & \frac{1}{3} & 0 & 1 & 1 \\
 \frac{1}{2} & \frac{1}{2} & 0 & \frac{1}{3} & 0 & 1 \\
 0 & \frac{1}{2} & 1 & \frac{1}{3} & 0 & 1 \\
 1 & \frac{1}{2} & 1 & \frac{1}{3} & 0 & 1 \\
\end{array}
\right), \quad
{\bf A}^{-1} =
\left(
\arraycolsep=2.2pt\def\arraystretch{1.2}
\begin{array}{cccccc}
 0 & 0 & 0 & 0 & -1 & 1 \\
 1 & -1 & 0 & 0 & 0 & 0 \\
 0 & 0 & 0 & -1 & \frac{1}{2} & \frac{1}{2} \\
 \frac{1}{2} & \frac{1}{2} & -1 & 0 & 0 & 0 \\
 \frac{1}{6} & \frac{1}{6} & \frac{2}{3} & -\frac{2}{3} & -\frac{1}{6} & -\frac{1}{6} \\
 -\frac{2}{3} & \frac{1}{3} & \frac{1}{3} & 1 & \frac{1}{2} & -\frac{1}{2} \\
\end{array}
\right).
$$
In this case $\det({\bf A})=1$ and $\vo(S)=\frac{|\det({\bf A})|}{5!}=\frac{1}{120}.$
Basic Lagrange polynomials for $S$ are
\begin{align*}
\lambda_1(x)&=  x_2+\frac{1}{2}\,x_4+\frac{1}{6}\,x_5-\frac{2}{3}, &
\lambda_2(x)&= -x_2+\frac{1}{2}\,x_4+\frac{1}{6}\,x_5+\frac{1}{3}, \\
\lambda_3(x) &= -x_4+\frac{2 }{3}\,x_5+\frac{1}{3} ,  &
 \lambda_4(x) &= -x_3-\frac{2 }{3}\,x_5+1, \\
  \lambda_5(x) &=-x_1+\frac{1}{2}\,x_3-\frac{1}{6}\,x_5+\frac{1}{2}, &
  \lambda_6(x) &= x_1+\frac{1}{2}\,x_3-\frac{1}{6}\,x_5-\frac{1}{2} .
\end{align*}
Let is find $\xi(S)$ using (\ref{nev_uhl_xi_s_cub_formula}). Computations show that
\begin{equation} \label{nev_uhl_eq23}
\max_{1\leq k\leq 6}
\max_{x\in \ver(Q_5)}(-\lambda_k(x))=\frac{2}{3} . 
\end{equation} 
It follows from (\ref{nev_uhl_xi_s_cub_formula}) that
$$\xi(S)=6\max_{1\leq k\leq 6}
\max_{x\in \ver(Q_5)}(-\lambda_k(x))+1=6\cdot \frac{2}{3} +1=5.
$$
Since $\xi(S)=5=\operatorname{dim}{\mathbb R}^5$, inequalities
 (\ref{nev_uhl_xi_geq_alpha_geq_n}) imply $\alpha(S)=\xi(S)=5$. 
It means that the simplex $5S$ is circumscribed around $Q_5$.
We get from (\ref{nev_uhl_alpha_d_i_formula}) that each axial diameter 
of $S$ is equal to $1$. 

Maximum in (\ref{nev_uhl_eq23}) is achieved as follows:
for $k=1$ --- on vertices  
$(0,0,0,0,0)$, $(0,0,1,0,0)$, $(1,0,0,0,0)$, $(1,0,1,0,0)$;
for $k=2$ --- on vertices
$(0,1,0,0,0)$, $(0,1,1,0,0)$, $(1,1,0,0,0)$, $(1,1,1,0,0)$;
for $k=3$ --- on vertices
$(0,0,0,1,0)$, $(0,0,1,1,0)$, $(0,1,0,1,0)$, $(0,1,1,1,0)$, $(1,0,0,1,0)$, $(1,0,1,1,0)$, $(1,1,0,1,0)$, $(1,1,1,1,0)$
for $k=4$ --- on vertices
$(0,0,1,0,1)$, $(0,0,1,1,1)$, $(0,1,1,0,1)$, $(0,1,1,1,1)$, $(1,0,1,0,1)$, $(1,0,1,1,1)$, $(1,1,1,0,1)$, $(1,1,1,1,1)$;
for $k=5$ ---  on vertices
$(1,0,0,0,1)$, $(1,0,0,1,1)$, $(1,1,0,0,1)$, $(1,1,0,1,1)$;
for $k=6$ --- on vertices
$(0,0,0,0,1)$, $(0,0,0,1,1)$, $(0,1,0,0,1)$, $(0,1,0,1,1)$.
Therefore, the condition (\ref{nev_uhl_circum_cub_cond}) is satisfied. 
It has the form
$$
\max\limits_{x\in \ver(Q_5)} \left(-\lambda_1(x)\right)=
\ldots=
\max\limits_{x\in \ver(Q_5)} \left(-\lambda_{6}(x)\right)=\frac{2}{3}.
$$
This condition provides an alternative proof of the fact that  simplex $5S$ is circumscribed around the cube $Q_5$.

The list of vertices delivering the maximum in (\ref{nev_uhl_eq23}) includes all vertices of the cube.
It means that $(n-1)$-dimensional faces of the simplex $5S$ (which form its boundary) contain 
all the vertices of the cube.

The theorem is proved.
\end{proof}

\begin{corollary*}\label{nev_uhl_corol_xi5is5}
$\xi_5=5$.
\end{corollary*}
\begin{proof} Let $S$ be the simplex from the proof of Theorem~\ref{nev-uhl-ksi5is5}. Then
$\xi_5\leq \xi(S)=5$. 
It follows from (\ref{nev_uhl_xi_geq_alpha_geq_n}) that the inverse inequality $\xi_5\geq 5$ holds. Therefore $\xi_5=5$. 
\end{proof}

It is easy to see that center of gravity of  $S$  is the point
$\left(
 \frac{1}{2},
 \frac{1}{2},
 \frac{1}{2},
 \frac{1}{2},
 \frac{1}{2}
\right)$. 
This fact also can be obtained from Theorem~\ref{nev_uhl_theor_center}.

The results of this section mean that {\it there exist perfect simplices in ${\mathbb R}^5$}.
Moreover, we have found infinite family of such simplices.  This family is described in the next section.

\section{Some Family of Perfect Simplices in ${\mathbb R}^5$}
\label{nevuhl_perf_fam5}

Consider the family of simplices $R=R(s,t)$ in ${\mathbb R}^5$ with vertices 
$\left(s,1,\frac{1}{3},1,1\right)$, 
$\left(s, 0,  \frac{1}{3}, 1, 1\right)$, $\left(s,2-3 t,\frac{1}{3},0,1\right)$,
$\left(2-3 s,t,0,\frac{1}{3},0\right)$, $\left(0, t, 1, \frac{1}{3}, 0\right)$, 
$\left(1, t, 1, \frac{1}{3}, 0\right)$ \linebreak
where $s,t\in {\mathbb R}$. 
The matrix ${\bf A}$ for the simplex $R$ has the form
$$
{\bf A}={\bf A}(s,t)=
\left(
\arraycolsep=3pt\def\arraystretch{1.3}
\begin{array}{cccccc}
 s & 1 & \frac{1}{3} & 1 & 1 & 1 \\
 s & 0 & \frac{1}{3} & 1 & 1 & 1 \\
 s & 2-3 t & \frac{1}{3} & 0 & 1 & 1 \\
 2-3 s & t & 0 & \frac{1}{3} & 0 & 1 \\
 0 & t & 1 & \frac{1}{3} & 0 & 1 \\
 1 & t & 1 & \frac{1}{3} & 0 & 1 \\
\end{array}
\right).
$$
An easy computation shows that  $\det ({\bf A})=1$.
So, for any $s$ and $t$, the simplex $R$ is nondegenerate and 
$$
\vo(R)=\frac{|\det({\bf A})|}{5!}=\frac{1}{120}.
$$

Now assume that $\frac{1}{3}\leq s,t\leq \frac{2}{3}$,
then $R\subset Q_5$.
The matrix ${\bf A}^{-1}$  has the form
$$
{\bf A}^{-1}={\bf A}^{-1}(s,t)=
\left(
\arraycolsep=5pt\def\arraystretch{1.4}
\begin{array}{cccccc}
 0 & 0 & 0 & 0 & -1 & 1 \\
 1 & -1 & 0 & 0 & 0 & 0 \\
 0 & 0 & 0 & -1 & 3 s-1 & 2-3 s \\
 2-3 t & 3 t-1 & -1 & 0 & 0 & 0 \\
 3 t-\frac{4}{3} & \frac{5}{3}-3 t & \frac{2}{3} & -\frac{2}{3} & 3 s-\frac{5}{3} & \frac{4}{3}-3 s \\
 -\frac{2}{3} & \frac{1}{3} & \frac{1}{3} & 1 & 2-3 s & 3 s-2 \\
\end{array}
\right).
$$
Let us write down the basic Lagrange polynomials of $R$:
\begin{equation}\label{nev_uhl_max_lagpol}
\arraycolsep=5pt\def\arraystretch{1.4}
\begin{array}{ll}
\lambda_1(x)&= x_2 + (2-3 t) x_4+\left(3 t-\frac{4}{3}\right) x_5-\frac{2}{3}, \\
\lambda_2(x)&= -x_2 + (3 t-1) x_4+ \left(\frac{5}{3}-3 t\right)x_5+\frac{1}{3}, \\
\lambda_3(x) &=-x_4+\frac{2 }{3} x_5+\frac{1}{3}  , \\
 \lambda_4(x) &=-x_3-\frac{2 }{3} x_5+1, \\
  \lambda_5(x) &=-x_1+(3 s-1) x_3+\left(3 s-\frac{5}{3}\right) x_5+2-3 s, \\
  \lambda_6(x) &=x_1+(2-3 s) x_3+\left(\frac{4}{3}-3 s\right) x_5+3 s-2 .
\end{array}
\end{equation}

The axial diameters of $R$ can be found by  (\ref{nev_uhl_d_i_formula}):
$$
d_1(s,t)=d_2(s,t)=1,
$$
$$
 d_3(s,t)=\frac{2}{\left| 1-3 s\right| +\left| 2-3 s\right| +1}, \quad  d_4(s,t)  = \frac{2}{\left| 1-3 t\right| +\left| 2-3 t\right| +1},  
$$
$$
d_5(s,t)=
\frac{2}{\left| \frac{4}{3}-3 s\right| +\left| \frac{5}{3}-3 s\right| +\left| \frac{4}{3}-3 t\right| +\left| \frac{5}{3}-3 t\right| +\frac{4}{3}}.
$$
Therefore,
$$
\alpha(s,t)=
\frac{1}{2} \left(\left| 1-3 s\right| +\left| 2-3 s\right| +\left| \frac{4}{3}-3 s\right| +\left| \frac{5}{3}-3 s\right| +  \left| 1-3 t\right|  + \right.
$$
$$
\left. 
+\left| 2-3 t\right| +\left| \frac{4}{3}-3 t\right| +\left| \frac{5}{3}-3 t\right| +\frac{22}{3}\right).
$$
If the conditions 
\begin{equation}\label{nev_uhl_xi_is_5_cond}
s\in\left[\frac{4}{9},\frac{5}{9}\right], \quad  t \in\left[\frac{4}{9},\frac{5}{9}\right]
\end{equation}
are satisfied, then we have $d_i(s,t)=1$ and $\alpha(s,t)=5$.

Define
$$
M(s,t):=
\max_{1\leq k\leq 6}\max_{x\in \ver(Q_5)}(-\lambda_k(x)).
$$
Calculations with the use of (\ref{nev_uhl_max_lagpol}) give
\begin{equation}\label{nev_uhl_max_st}
M(s,t)=\max \left\{\frac{2}{3}, \,2-3 s, \,3 s-1, \,2-3 t, \,3 t-1\right\}.
\end{equation}
Now we can find  $\xi(R)$. We have
$$
\arraycolsep=5pt\def\arraystretch{1.4}
\begin{array}{rl}
M(s,t)=\frac{2}{3}, & s\in\left[\frac{4}{9},\frac{5}{9}\right], \  t \in\left[\frac{4}{9},\frac{5}{9}\right];\\
M(s,t)>\frac{2}{3}, & s\notin\left[\frac{4}{9},\frac{5}{9}\right], \ t \notin\left[\frac{4}{9},\frac{5}{9}\right]. \\
\end{array}
$$
Applying (\ref{nev_uhl_xi_s_cub_formula}), we obtain
\begin{equation}\label{nev_uhl_xi_of_st}
\arraycolsep=5pt\def\arraystretch{1.4}
\begin{array}{rl}
\xi(R)=6M(s,t)+1=5, & s\in\left[\frac{4}{9},\frac{5}{9}\right], \  t \in\left[\frac{4}{9},\frac{5}{9}\right];\\
\xi(R)=6M(s,t)+1>5, & s\notin\left[\frac{4}{9},\frac{5}{9}\right], \ t \notin\left[\frac{4}{9},\frac{5}{9}\right]. \\
\end{array}
\end{equation}

The graph of the function $\xi=\xi(R(s,t))$ for $s\in\left[\frac{1}{3},\frac{1}{2}\right], \  t \in\left[\frac{1}{3},\frac{2}{3}\right]$ is shown in Fig.~\ref{fig:nev_ukl_xi}.
\begin{figure}
\includegraphics[width=1.0\textwidth]{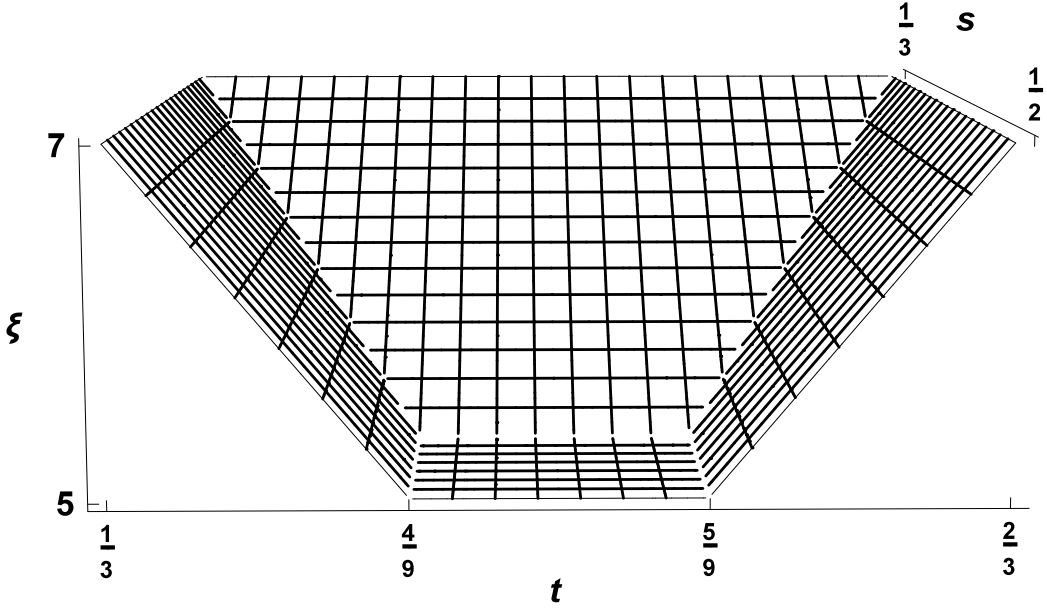}
\caption{Graph of the function  $\xi=\xi(R(s,t))$ for $s\in\left[\frac{1}{3},\frac{1}{2}\right], \  t \in\left[\frac{1}{3},\frac{2}{3}\right]$}
\label{fig:nev_ukl_xi}
\end{figure}
It follows from (\ref{nev_uhl_xi_of_st}) and Collorary~\ref{nev_uhl_corol_xi5is5} that,
if the conditions (\ref{nev_uhl_xi_is_5_cond}) are satisfied, then 
 $\xi(R(s,t))=\xi_5=5$. Thus, $R(s,t) \subset Q_5 \subset 5 R(s,t)$.

\begin{table}
\caption{Main extremal vertices}
\label{tab:nev_uhl_famvert_base}       
\bgroup
\def\arraystretch{2.0}%
\begin{tabular}{|c|c|}
\hline 
$k$ & 
  \def\arraystretch{1.5}
   \begin{tabular}{c}
    Vertices of $Q_5$ such that $(-\lambda_k(x))=\frac{2}{3}$ \\ 
    for $\frac{4}{9}\leq s \leq \frac{5}{9}$, $\frac{4}{9}\leq t \leq \frac{5}{9}$
       \end{tabular}
   \\
\hline 
$1$   &  
  \def\arraystretch{1.3}
   \begin{tabular}{c}
    $(0,0,0,0,0)$,   $(0,0,1,0,0)$,  $(1,0,0,0,0)$,    $(1,0,1,0,0)$ \\ 
       \end{tabular}
         \\
  \hline 
 $2$   &  
  \def\arraystretch{1.3}
   \begin{tabular}{c}
    $(0,1,0,0,0)$, $(0,1,1,0,0)$, $(1,1,0,0,0)$, $(1,1,1,0,0)$ \\ 
       \end{tabular}
        \\
  \hline
 $3$    &
  \def\arraystretch{1.3}
   \begin{tabular}{c}
    $(0,0,0,1,0)$, $(0,0,1,1,0)$, $(0,1,0,1,0)$, $(0,1,1,1,0)$,\\ $(1,0,0,1,0)$, $(1,0,1,1,0)$, $(1,1,0,1,0)$, $(1,1,1,1,0)$ \\ 
       \end{tabular}
             \\
  \hline
$4$   & 
  \def\arraystretch{1.3}
   \begin{tabular}{c}
    $(0,0,1,0,1)$, $(0,0,1,1,1)$, $(0,1,1,0,1)$, $(0,1,1,1,1)$,\\ $(1,0,1,0,1)$, $(1,0,1,1,1)$, $(1,1,1,0,1)$, $(1,1,1,1,1)$ \\ 
       \end{tabular}
          \\
  \hline
 $5$   &  
  \def\arraystretch{1.3}
   \begin{tabular}{c}
    $(1,0,0,0,1)$, $(1,0,0,1,1)$, $(1,1,0,0,1)$, $(1,1,0,1,1)$ \\ 
       \end{tabular}
      \\
  \hline
 $6$   &  
  \def\arraystretch{1.3}
   \begin{tabular}{c}
    $(0,0,0,0,1)$, $(0,0,0,1,1)$,  $(0,1,0,0,1)$, $(0,1,0,1,1)$ \\ 
       \end{tabular}
         \\
  \hline
 \end{tabular}
\egroup
\end{table}


\begin{table}
\caption{Additional extremal vertices}
\label{tab:nev_uhl_famvert_special}       
\bgroup
\def\arraystretch{2.0}%
\begin{tabular}{|c|c|c|}
\hline 
$k$ & Values of $s$ and $t$ &   
\def\arraystretch{1.5}
   \begin{tabular}{c}
    Vertices of $Q_5$ such that $(-\lambda_k(x))=\frac{2}{3}$ \\ 
    only for given values of $s$ and $t$
       \end{tabular} \\
\hline 
 $1$ & $\frac{4}{9}\leq s \leq \frac{5}{9}$, $t=\frac{4}{9}$   &  
            \def\arraystretch{1.3}
           \begin{tabular}{c}
          $(0,0,0,0,1)$, $(0,0,1,0,1)$, \\ $(1,0,0,0,1)$, $(1,0,1,0,1)$\\
          \end{tabular}
          \\
  \hline 
   $2$ & $\frac{4}{9}\leq s \leq \frac{5}{9}$, $t=\frac{5}{9}$  &  
           \def\arraystretch{1.3}
           \begin{tabular}{c}
          $(0,1,0,0,1)$, $(0,1,1,0,1)$, \\ $(1,1,0,0,1)$, $(1,1,1,0,1)$\\
          \end{tabular}
          \\
  \hline
 $5$ &  $s=\frac{5}{9}$, $\frac{4}{9}\leq t \leq \frac{5}{9}$   &  
           \def\arraystretch{1.3}
           \begin{tabular}{c}
          $(1,0,0,0,0)$, $(1,0,0,1,0)$, \\ $(1,1,0,0,0)$, $(1,1,0,1,0)$\\
          \end{tabular}
          \\
  \hline
 $6$ & $s=\frac{4}{9}$, $\frac{4}{9}\leq t \leq \frac{5}{9}$   &  
            \def\arraystretch{1.3}
           \begin{tabular}{c}
          $(0,0,0,0,0)$, $(0,0,0,1,0)$, \\ $(0,1,0,0,0)$, $(0,1,0,1,0)$\\
          \end{tabular}
          \\
  \hline
 \end{tabular}
\egroup
\end{table}

In the following we assume that  (\ref{nev_uhl_xi_is_5_cond}) is true.
Since $\xi(R)=5=\operatorname{dim}{\mathbb R}^5$, we have $\alpha(R)=\xi(R)$
(see section~\ref{nev_uhl_p_intro}). Therefore, the simplex $5R$ is circumscribed around $Q_5$.
The equality $\alpha(R)=5$ implies that axial diameters of  $R$ are equal to $1$. 
If  (\ref{nev_uhl_xi_is_5_cond}) holds true, then (\ref{nev_uhl_max_st}) takes the form
\begin{equation}\label{nev_uhl_max_23}
\max_{1\leq k\leq 6} \ \max_{x\in \ver(Q_5)}(-\lambda_k(x))=\frac{2}{3}.
\end{equation}
The set of vertices of $Q_5$ delivering maximum in (\ref{nev_uhl_max_23}),
for given $k$, depends on $s$ and $t$. 
Let us study the behavior of the coefficients of basic Lagrange polynomials 
(\ref{nev_uhl_max_lagpol}).
If $s$ and $t$ satisfy (\ref{nev_uhl_xi_is_5_cond}), then 
$2-3t>0$, $3t-1>0$, $3s-1>0$, $2-3s>0$. 
For the rest of coefficients depending on  $s$ and $t$, we have inequalities
$3 t-\frac{4}{3}\geq 0$, $\frac{5}{3}-3 t\geq 0$, $3 s-\frac{5}{3}\leq 0$, $\frac{4}{3}-3s\leq 0$. 
These expressions can be equal to $0$ only on the boundary of the area (\ref{nev_uhl_xi_is_5_cond}).
If   the multiplier in front of $x_i$ is equal to $0$, then
$x_i$ does not influence the value of $\lambda_k(x)$. 
Since, for $x\in \ver ( Q_5)$, the variable  $x_i$ takes 
values $0$ or $1$, the number of extremal vertices for  $\lambda_k(x)$ in this situation doubles.

The lists of vertices delivering maximum in (\ref{nev_uhl_max_23}) are given in
Table~\ref{tab:nev_uhl_famvert_base} and Table~\ref{tab:nev_uhl_famvert_special}.
Denote $\Pi=\left\{(s,t): \frac{4}{9}<s< \frac{5}{9}, \ \frac{4}{9}< t< \frac{5}{9}   \right\}$.
Using Theorem~\ref{nev_uhl_lem_vertices}, we obtain  the following result. 
For any $(s,t)\in\Pi$, 
the set $\ver ( Q_5)$ is divided into 6 non-intersected classes of extremal vertices.
Number of all extremal vertices is equal to 32,  i.\,e., the number of vertices of $Q_5$.
If  $(s,t)\in\Pi$, then the vertices of the same class are the inner points of the same 4-dimensional
facet of the simplex $5R(s,t)$. 
The partition of $\ver(Q_5)$ into classes is represented in  Table~\ref{tab:nev_uhl_famvert_base}.

Additionally, if $k=1,2,5,6$, then, for some values of  $(s,t)$
(in particular, when  $(s,t)\in\partial\Pi$),
some other vertices also will be extremal.
All such cases are listed in Table~\ref{tab:nev_uhl_famvert_special}.
In each of these cases, corresponding vertices belong to two different 4-dimensional facets 
of $5R(s,t)$, i.\,e.,  belong to the facets of smaller dimension.

If  one of  $s$ and  $t$ is the endpoint of 
 $\left[ \frac{4}{9},  \frac{5}{9}\right]$, then the number of extremal vertices
  (considering repetitions) is equal to 36. 
 If both $s$ and  $t$ are the endpoints of the segment, then the number of extremal vertices
  (considering repetitions) is equal to 40.

It is easy to see that the center of gravity of $R(s,t)$ coincides with the center of gravity of $Q_5$.
This fact could be also derived from Theorem~\ref{nev_uhl_theor_center}.
\section{Nonregular Extremal Simplices in ${\mathbb R}^7$}
\label{nev_uhl_p_fam7}

Since $8$ is an Hadamard number, it follows from  Theorem~\ref{nev-uhl-adamartheorem} that there exist the regular simplex  $S\subset Q_7$ such that
$\xi(S)=\xi_7=7$. 
In this section we show that the property $\xi(S)=\xi_7$ holds also for some nonregular simplices.
In this sense the case $n=7$ is  analogous to the case $n=3$ (see section~\ref{nev_uhl_p_n13}).

Let  $T=T(t)$ be the simplex with vertices
 $(1, 0, 0, 0, 0, 0, 1)$, 
 $(1, 0, 1, t, 1, 1, 0)$, \linebreak
  $(0, 1, 1, 1 - t, 0, 1, 1)$,  
 $(0, 0, 0, t, 1, 1, 0)$,  
 $(0, 1, 1, 1 - t, 0, 0, 0)$,   
 $(1, 1, 0, t, 1, 1, 0)$, \linebreak
 $(0, 1, 1, 1 - t, 1, 0, 1)$, 
  $(1, 0, 0, 1, 0, 0, 1)$.
  If $0\leq t \leq 1$,  then   $T(t)\subset Q_7$.
The matrix ${\bf A}(t)$  for the simplex $T$ has the form
$$
{\bf A} (t)=
\left(
\arraycolsep=3.3pt\def\arraystretch{1.0}
\begin{array}{cccccccc}
 1 & 0 & 0 & 0 & 0 & 0 & 1 & 1 \\
 1 & 0 & 1 & t & 1 & 1 & 0 & 1 \\
 0 & 1 & 1 & 1-t & 0 & 1 & 1 & 1 \\
 0 & 0 & 0 & t & 1 & 1 & 0 & 1 \\
 0 & 1 & 1 & 1-t & 0 & 0 & 0 & 1 \\
 1 & 1 & 0 & t & 1 & 1 & 0 & 1 \\
 0 & 1 & 1 & 1-t & 1 & 0 & 1 & 1 \\
 1 & 0 & 0 & 1 & 0 & 0 & 1 & 1 \\
\end{array}
\right).
$$
For any $t$, holds  $\det \left({\bf A}(t) \right)=8$.
Therefore, $T(t)$ is nondegenerate and
$$
\vo (T(t)) = \frac{|\det ({\bf A}(t))|}{7!}  = \frac{1}{630} = 0.001587\dots
$$

\begin{theorem}\label{nev-uhl-ksi7is7family}
If  $t\in \left[\frac{1}{4},\frac{3}{4}\right]$, then  simplex $7T(t)$ is circumscribed around $Q_7$.
\end{theorem}
\begin{proof}
The matrix  $[{\bf A}(t)]^{-1}$  has the form
$$
[{\bf A}(t)]^{-1} = \frac{1}{8}
\left(
\arraycolsep=3.5pt\def\arraystretch{1.0}
\begin{array}{cccccccc}
 4 t-1 & 3 & -1 & -5 & -1 & 3 & -1 & 3-4 t \\
 1-4 t & -3 & 1 & -3 & 1 & 5 & 1 & 4 t-3 \\
 1-4 t & 5 & 1 & -3 & 1 & -3 & 1 & 4 t-3 \\
 -8 & 0 & 0 & 0 & 0 & 0 & 0 & 8 \\
 4 t-3 & 1 & -3 & 1 & -3 & 1 & 5 & 1-4 t \\
 4 t-3 & 1 & 5 & 1 & -3 & 1 & -3 & 1-4 t \\
 3-4 t & -1 & 3 & -1 & -5 & -1 & 3 & 4 t-1 \\
 6 & -2 & -2 & 6 & 6 & -2 & -2 & -2 \\
\end{array}
\right).
$$
The columns of this matrix contain the coefficients of basic Lagrange polynomials $\mu_k(x)$, $1\leq k \leq 8$.
Using  (\ref{nev_uhl_d_i_formula}) and (\ref{nev_uhl_alpha_d_i_formula}),  we obtain 
$$
d_i(T)=\frac{16}{\left| 1-4 t\right| +\left| 3-4 t\right| +14}, \quad \ i\neq 4; \qquad d_4(T)=1; 
$$
$$
\alpha(T) = 
\frac{1}{8} \bigl(3 \left| 1-4 t\right| +3 \left| 3-4 t\right| +50 \bigr) = 
 \left\{ 
 \arraycolsep=3pt\def\arraystretch{1.2} 
\begin{array}{cc}
 7, & \frac{1}{4}\leq t\leq \frac{3}{4}, \\
 \frac{31}{4}-3 t, & t<\frac{1}{4}, \\
 3 t+\frac{19}{4}, & t>\frac{3}{4}. \\
\end{array}
\right. 
$$
For $t\in \left[\frac{1}{4},\frac{3}{4}\right]$, holds  $d_i(T)=1$. Therefore,
$\alpha(T) = \sum_{i=1}^7 \frac{1}{d_i(T)} = 7$.
\begin{figure}
\includegraphics[width=1\textwidth]{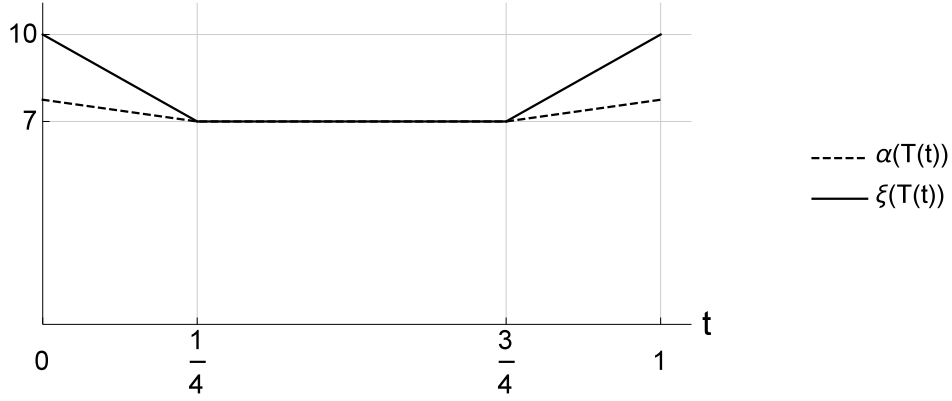}
\caption{The graphs of the functions $\alpha(T(t))$ and  $\xi(T(t))$ }
\label{fig:nev_ukl_ksi7_is_7_family}
\end{figure}

Now, using (\ref{nev_uhl_xi_s_cub_formula}), we compute  $\xi(T)$ :
$$
\xi(T)= 8 \max_{1\leq k\leq 8} \ \max_{x\in \ver Q_7} (-\mu_k(x)) +1  
=  
 \left\{  
\begin{array}{cc}
 7, & \frac{1}{4}\leq t\leq \frac{3}{4}, \\
 10-12 t, & t<\frac{1}{4}, \\
 12 t-2, & t>\frac{3}{4}. \\
\end{array}
\right.
$$

Consequently, for $\frac{1}{4} \leq t \leq \frac{3}{4}$, we have $\alpha(T)=\xi(T)=7$. 
It means that simplex $7T$ is circumscribed around  $Q_7$. 
The theorem is proved.
\end{proof}

The graphs of the functions $\alpha(T(t))$ and  $\xi(T(t))$ are shown in  Fig.~\ref{fig:nev_ukl_ksi7_is_7_family}.


\section{The Exact Value of $\xi_9$}
\label{nev_uhl_p_fam9}

In this section we  show that $\xi_9=9$.
Moreover, we describe some family of simplices $S\subset Q_9$ such that $\xi(S)=9$.

Let  $S=S(t)$ be the simplex with vertices
$(1, 0, 0, 0, 0, 0, 0, 0, 1)$, \linebreak
$(1, 1, 0, 1, t, 1, 1, 0, 0)$,
$(1, 0, 1, 1, 1 - t, 0, 1, 1, 0)$,
$(0, 1, 1, 1, t, 0, 0, 1, 1)$, \linebreak
$(0, 1, 1, 0, 1 - t, 1, 0, 0, 0)$,
$(0, 0, 0, 1, t, 0, 1, 1, 0)$,
$(1, 1, 0, 0, 1 - t, 1, 1, 1, 0)$, \linebreak
$(0, 1, 1, 0, t, 1, 1, 0, 1)$,
$(0, 0, 1, 1, 1 - t, 1, 0, 1, 1)$,
$(1, 0, 0, 0, 1, 0, 0, 0, 1)$.  \linebreak
If  $0 \leq t \leq 1$, then $S(t)\subset Q_9$.
Omitting the details, we present characteristics of $S$:
$$
\det({\bf A})=25, \qquad \vo(S)=\frac{|\det({\bf A})|}{9!} = \frac{5}{72576} = 0.00006889\dots
$$
$$
d_i(S) = \frac{50}{\left| 10-25 t\right| +\left| 15-25 t\right| +45}, \  i\neq 5; \quad d_5 =1;
$$
$$
\alpha(S)=
 \left\{
 \arraycolsep=3pt\def\arraystretch{1.2} 
\begin{array}{cc}
 9, & \frac{2}{5}\leq t\leq \frac{3}{5}, \\
 \frac{61}{5}-8 t, &  t<\frac{2}{5}, \\
8 t+\frac{21}{5}, & t>\frac{3}{5}, \\
\end{array}
\right. \qquad
\xi(S) = 
 \left\{ 
 \arraycolsep=3pt\def\arraystretch{1.2}
\begin{array}{cc}
 9, & \frac{2}{5}\leq t\leq \frac{3}{5}, \\
 25-40 t, &  t<\frac{2}{5}, \\
 40 t - 15, & t>\frac{3}{5}. \\
\end{array}
\right.
$$
If $\frac{2}{5}\leq t\leq \frac{3}{5}$, then $\alpha(S)=\xi(S)=9$. 
Since  $\xi_9\geq 9$ (see (\ref{nev_uhl_xi_geq_alpha_geq_n})), we get $\xi_9=9$.
It follows from the equality $\alpha(S)=\xi(S)$ that  simplex
$9S$ is circumscribed around  $Q_9$.
Additional analysis shows that simplex $S(t)$ is not perfect. 

\section{Concluding Remarks}
\label{nev_uhl_p_concluding}

The values $\alpha(S)$, $\xi(S)$, $d_i(S)$, $\xi_n$ are connected with some geometric estimates in polynomial interpolation. Detailed description of this connection can
be found in the book \cite{nu_bib_3}.
These questions are beyond the scope of this paper.

Our results mean that, for all odd  $n$ from the interval $1\leq n \leq 9$, holds
 $\xi_n=n$.
 This equality holds also for the infinite set of numbers $n$ such that $n+1$ is an Hadamard number.
We suppose that $\xi_n=n$ for all odd $n$. 
However, by the time of submitting the paper, we have nor proof, nor refutation of this conjecture.

As it was mentioned above, perfect simplices in  ${\mathbb R}^n$ exist at least for $n=1,3,5$. 
It is unknown whether there exist perfect simplices for $n>5$.
It is also unclear whether there exist perfect simplices 
in  ${\mathbb R}^n$ for any even $n$. 
Note that for $n=2$ perfect simplices do not exist. 
This result follows from the  full description of simplices $S\subset Q_2$ such that
$\xi(S)=\xi_2$, see \cite{nu_bib_3}. 

The examples of extremal simplices for $n=5,7,9$ were found partially with the use of computer. 
The families of simplices were constructed by generalization of individual examples. 
In particular, we utilized the property of simplices from Theorem~\ref{nev_uhl_theor_center}.
The figures were created with application of the system Wolfram Mathematica [9].
Despite the fact that we used numerical computations, all the results of this paper are exact
and the proofs are complete.


\end{document}